\documentclass[11pt,reqno]{amsart}

\usepackage{amsmath}
\usepackage{amssymb}
\usepackage{amsthm}
\usepackage{amscd, amsfonts, mathrsfs}

\usepackage{cases}

\usepackage[dvips]{epsfig}
\usepackage{verbatim} 

\usepackage{epsf}
\usepackage[bookmarksnumbered,pdfpagelabels=true,plainpages=false,colorlinks=true,
            linkcolor=black,citecolor=black,urlcolor=blue]{hyperref}

\theoremstyle{plain}
\newtheorem{theorem}{Theorem}[section]
\newtheorem{lemma}[theorem]{Lemma}
\newtheorem{prop}[theorem]{Proposition}

\theoremstyle{definition}
\newtheorem{rema}[theorem]{Remark}
\newtheorem{notation}[theorem]{Notation}

\newtheorem{defi}[theorem]{Definition}

\numberwithin{equation}{section}

\newcommand{\cC}{\mathcal C}

\newcommand{\cQ}{\mathcal Q}

\newcommand{\al}{\alpha}

\newcommand{\de}{\delta}

\newcommand{\la}{\lambda}

\newcommand{\si}{\sigma}

\newcommand{\RR}{\mathbb R}

\newcommand{\cqd}{\hfill $\qed$\\ \medskip}
\newcommand{\rar}{\rightarrow}

\newcommand{\ve}{\varepsilon}

\newcommand{\p}{\parallel}

\allowdisplaybreaks

\title[Critical Schr\"odinger-Newton]{Some a priori estimates
for a critical Schr\"odinger-Newton equation}

\author[Disconzi]{Marcelo M. Disconzi}
\address{Department of Mathematics\\
Vanderbilt University, Nashville, TN 37240, USA}
\email{marcelo.disconzi@vanderbilt.edu}

\begin{document}

\begin{abstract}
Under natural energy and decay assumptions, 
we derive a priori estimates for solutions
of a Schr\"odinger-Newton type of equation with critical exponent.
On one hand, such an equation generalizes the traditional 
Schr\"odinger-Newton and Choquard equations, while, on the other hand,
it is naturally related to problems involving scalar curvature and
conformal deformation of metrics.
\end{abstract}

\maketitle


\section{Introduction. \label{intro}}

We shall study the behavior of positive solutions to the equation
\begin{gather}
\Delta u + \Big( \frac{1}{\,|x|^\ell}  * u^\frac{ 2n }{n-2} \Big) u^\frac{n+2}{n-2}
- V u = 0
\label{critical_SN_eq}
\end{gather}
in $\RR^n$, $n \geq 3$,  where $\Delta$ is the Euclidean Laplacian, $*$ means convolution, $\ell$ is 
a real number, and
$V$ is a smooth real valued function. Equation (\ref{critical_SN_eq}) will be referred to
as 
\emph{critical
Schr\"odinger-Newton equation}. We are concerned with a priori estimates, i.e., 
bounds on $u$ and its derivatives, which are necessary conditions 
for any (positive) solution of (\ref{critical_SN_eq}). Needless to say, not only 
are  a priori estimates one of the main tools 
 towards an existence theory for a 
given PDE, but also they reveal important features about the behavior of solutions.
 The form of the bounds one generally seeks to establish depends, of course, 
on specific characteristics of the equation, which in 
the case of (\ref{critical_SN_eq}), will be captured by suitable
hypotheses on  $\ell$, $V$ and asymptotic conditions for  $u$. 
In order to state natural assumptions for the critical  Schr\"odinger-Newton equation
as well as to highlight why one would consider (\ref{critical_SN_eq}) in the first place,
we first turn our attention to some related problems.

Recall that the Schr\"odinger-Newton equation is given by
\begin{gather}
i \hbar \frac{ \partial \Psi}{\partial t} = -\frac{\hbar^2}{2 m} \Delta \Psi
- Gm^2 \Big( \frac{1}{|x|} * |\Psi|^2 \Big ) \Psi, \, \,  x \in \RR^3,
\label{SN_eq}
\end{gather}
where $\Psi = \Psi(t,x)$ is a function on $\RR \times \RR^3$, and $m$, $\hbar$ 
and $G$ are constants. Physically, $\Psi$ is the wave-function 
of a self-gravitating quantum system of mass $m$ with gravitational 
interaction given by Newton's law of gravity; $\hbar$ and $G$
are Planck's and Newton's constants, respectively.
Equation (\ref{critical_SN_eq}) is obtained by considering the 
Schr\"odinger equation
\begin{gather}
i \hbar \frac{ \partial \Psi}{\partial t} = -\frac{\hbar^2}{2 m} \Delta \Psi
+  m U  \Psi
\label{Schrodinger_eq}
\end{gather}
with a Newtonian gravitational potential  $U$ that is sourced by a distribution
of mass given by $\Psi$ itself, 
\begin{gather}
\Delta U = 4 \pi G m |\Psi|^2.
\label{Poisson_eq}
\end{gather}
In other words, the mass distribution is given in probabilistic terms,
with its probability amplitude evolving according to the Schr\"odinger
equation, as is usual in quantum mechanical systems.
Writing $U$ in terms of the right hand side of  (\ref{Poisson_eq}) 
and the fundamental solution of the Laplacian, and using the resulting
expression into (\ref{Schrodinger_eq}), formally produces (\ref{SN_eq}).

The Schr\"odinger-Newton equation was first introduced by Ruffini
and Bonazzola in their study of equilibrium of self-gravitating 
bosons and spin-half fermions  \cite{RuffBo} and gained notoriety 
with Penrose's ideas about the role of gravity in the collapse of the 
wave function \cite{Penrose1, Penrose2}. More recently, it was used 
in discussions of semi-classical quantum gravity \cite{Carlip1,  Giu, Carlip2}.

Notice that the power of the convoluted term $\frac{1}{|x|}$ becomes $n-2$ in higher dimensions. 
Parallel to this situation, the following generalization of (\ref{SN_eq})
has been considered,
\begin{gather}
i \frac{ \partial \varphi}{\partial t} = \Delta \varphi
+ p \Big( \frac{1}{\, |x|^\ell}  * |\varphi|^p \Big ) \varphi |\varphi|^{p-2}, \, \, x \in \RR^n,
\label{Choquard_eq}
\end{gather}
where $p \geq 2$ and $\ell \in (0,n)$. For the remainder of the paper,  as in (\ref{Choquard_eq}), dimensional 
constants such as $\hbar$ and $G$ are
 set to one.
Equation (\ref{Choquard_eq}) is
used in certain approximating regimes of the Hartree-Fock theory for a 
one component plasma;
see e.g. \cite{Lieb1,Lieb2}.  In three spatial dimensions, with $\ell = 1$ and $p = 2$,
equation (\ref{Choquard_eq}) has been extensively studied, see
\cite{Barg,Caz, GiNir, Gini,  Lions1, Lions2, Rabi, ShaStrauss} and 
references therein.

As in many situations in Physics, one is particularly interested in 
wave-front-like solutions of the form $\varphi(t,x) = e^{i\omega t}u(x)$,
$\omega \in \RR$, which, 
upon plugging into (\ref{Choquard_eq}), leads to
\begin{gather}
\Delta u + \omega u 
+ p \Big( \frac{1}{\, |x|^\ell} * |u|^p \Big ) u | u|^{p-2} = 0.
\nonumber
\end{gather} 
 Considering the more general situation where $\omega \mapsto -V = -V(x)$ and
 dropping the factor $p$ in front of the convolution, we find
 \begin{gather}
\Delta u - V u 
+ \Big( \frac{1}{\, |x|^\ell} * |u|^p \Big ) u |u|^{p-2} = 0,
\label{stationary_Choquard_eq}
\end{gather} 
which is referred to as generalized non-linear Choquard equation.
A detailed study of  (\ref{stationary_Choquard_eq}), including
 existence results, has been recently carried out
by Ma and Zhao \cite{MaZhao};
Cingolani, Clapp  and Secchi \cite{Cing};  
Clapp and Salazar \cite{ClappSal}; and
Moroz and van
Schaftingen \cite{MorozSch}.
These works deal with the case where the exponent 
$p$ is sub-critical, i.e., $p < \frac{2n}{n-2}$, while the case $p=\frac{2n}{n-2}$
is called critical. Criticality 
is here understood in the usual sense of the Sobolev embedding theorems.
We recall that, roughly speaking, equations with sub-critical 
non-linearity are suited for treatment via calculus of variation 
techniques (see e.g. \cite{AubinBook}), provided
that the equation can be derived from an action principle --- which is 
the case for (\ref{stationary_Choquard_eq}), see the above references.

Equations with critical exponent appear in several situations in 
Physics and Mathematics (see e.g. \cite{AubinBook, C, Jost} and references therein).
One important case is the Yamabe equation 
\begin{gather}
\Delta_g u - \frac{n-2}{4(n-1)} R_g u + K u^\frac{n+2}{n-2} = 0, \, \, u > 0, 
\label{Yamabe_eq}
\end{gather}
where $\Delta_g$ and $R_g$ are, respectively, the Laplace-Beltrami
operator and the scalar curvature of a given metric $g$, and $K$ is a constant.
Equation (\ref{Yamabe_eq}) figures in the 
famous Yamabe problem \cite{Ya}. We recall that this
corresponds to finding a constant scalar curvature metric in the conformal 
class of a given closed\footnote{The Yamabe problem
for manifolds with boundary was studied in
\cite{A,Al,Al2,BC,Chen,DK_Yamabe,Es,Es2,Es3,FA,FA2,HL,HL2,M2,M3}.} $n$ dimensional
($n\geq 3$) Riemannian 
manifold. The complete solution of the Yamabe problem
through the works of Yamabe \cite{Ya}, Trudinger \cite{Tr}, Aubin
\cite{Au} and Schoen \cite{S1} 
was probably the first instance of a satisfactory existence theory for equations
 with critical 
non-linearity (see
\cite{LP} for a complete overview). The analogous equation for the Euclidean metric 
was studied in great detail by Caffarelli, Gidas and Spruck
\cite{CGS}.

Interest in the Yamabe problem has not faded with its resolution. On the 
contrary, the discovery of Pollack, that it is possible to find an
arbitrary  large number of solutions to the Yamabe equation
on manifolds with positive Yamabe invariant,  has led  to an intensive investigation of the properties of 
the space $\Phi$ of solutions 
to (\ref{Yamabe_eq}) --- see
\cite{LP} for 
a definition of the Yamabe invariant and \cite{Po} for a precise statement of
Pollack's result. A quite satisfactory account of the topology of $\Phi$
was given through the combined works of Khuri, Marques and Schoen \cite{KMS};
Brendle \cite{Br}; and Brendle and Marques \cite{BM} 
(see also \cite{Dr, LZh1,LZh2, M, S0, SZ} for earlier results). These 
results imply that $\Phi$ is compact in the $C^{2,\al}$ topology for $n\leq 24$
and non-compact otherwise\footnote{This under the assumption that the Yamabe invariant
is positive, and the underlying manifold is not conformally equivalent
to the round sphere. The cases of negative and zero Yamabe invariant are trivial. The 
geometric reasons for singling out the sphere, and the relation between the 
compactness of $\Phi$ and the geometry of the manifold, are discussed in
\cite{BM2, Ob}.}. Such results were extended to manifolds with boundary 
in \cite{DK_Yamabe}.

From an analytic perspective, the richness surrounding equation (\ref{Yamabe_eq}),
including the surprising cut-off in dimension $n=24$, is a direct consequence
of the critical exponent. It should be expected, therefore, that allowing
$p = \frac{2n}{n-2}$  in (\ref{Choquard_eq}) will lead to many interesting new phenomena, 
adding
to the already sophisticated nature of the generalized non-linear 
Choquard equation. 
A contribution in this direction is the goal of the present work.
In order not to lose sight of the relation between what has been just described
and our objectives in the rest of the manuscript, notice that from the point of view of the theory 
of partial 
differential equations, the aforementioned compactness of $\Phi$ 
corresponds to a priori bounds for solutions 
of (\ref{Yamabe_eq}). The reader should also notice the similarities 
between (\ref{critical_SN_eq}) and (\ref{Yamabe_eq}), specially if we are given a metric
$g$ in $\RR^n$, with $\Delta_g$ replacing $\Delta$ and 
$V$ being the scalar curvature.

We shall present a priori estimates for positive solutions to (\ref{critical_SN_eq}).
Such estimates constitute the first step towards an existence theory for 
this equation. They also provide insight on the structure of the space of solutions
to (\ref{critical_SN_eq}), at least for those solutions satisfying some additional 
requirements. We also give an account of the profile of blowing-up solutions.
We point out that since, to the best of our knowledge, equation (\ref{critical_SN_eq})
has not been considered before in the literature, we shall not attempt 
to derive very general results; rather, our focus will be on conditions that allow, on one hand,
a good grasp on the behavior of $u$ without, on the other hand, rendering the
problem uninteresting. We also stress that our methods may shed new light in the 
study of equations (\ref{SN_eq}) and (\ref{Choquard_eq}), in that we investigate the pointwise
behavior of solutions as opposed to the $L^2$ techniques previously employed to deal with 
these equations.

\section{Setting and statement of the results}

\begin{notation}
From now on, $u$ will denote a positive solution of (\ref{critical_SN_eq}).
\end{notation}

The first thing we investigate is
the range of $\ell$ values which will be allowed. For the integral
\begin{gather}
\frac{1}{\, |x|^\ell}* u^\frac{ 2n }{n-2} 
= \int_{\RR^n}\frac{1}{\, |y|^\ell} \, u^\frac{2n}{n-2}(x-y) \, dy
\label{convolution}
\end{gather}
to be finite near the origin without expecting any vanishing of $u$ in the neighborhood
of zero, we must have $\ell \in (0,n)$. Next, we ask 
 what kind of asymptotic behavior should be
required for $u$. Experience with equations with critical exponent 
\cite{CGS,LP} suggests that we should adopt
\begin{gather}
u = O(|x|^{2-n}) \, \text{ as } \ |x| \rar \infty. 
\label{asymptotic}
\end{gather}
Then (\ref{asymptotic}) and $\ell \in (0,n)$ guarantee that the integral
(\ref{convolution}) is finite.

\begin{defi} Given real numbers $\varrho > 0$ and $L > 0$, we say that 
$u$ has $(\varrho, L)$-\emph{decay}  if it satisfies
\begin{gather}
u(x) \leq L |x|^{2-n} , \, \text{ for } |x| \geq \varrho.
\nonumber
\end{gather}
Denote by $\cC_{\varrho,L}$ the set of solutions $u$ with $(\varrho, L)$-decay.

\end{defi}
We shall also need some energy conditions. In order to motivate them, multiply
(\ref{critical_SN_eq}) by $u$, integrate by parts, and assume that all the integrals are finite.
Then
\begin{gather}
\int_{\RR^n} \big ( |\nabla u |^2 +V u^2 \big ) = 
\int_{\RR^n} \int_{\RR^n} \frac{1}{\, |y|^\ell} \, u^\frac{2n}{n-2}(x-y) u^\frac{2n}{n-2} (x)\, dy \, dx.
\label{energy}
\end{gather}
The left-hand side is just the energy associated with the linear operator
$\Delta - V$. If we had a constant rather than $\frac{1}{ \, |x|^\ell} * u^\frac{ 2n }{n-2} $,
the above expression would produce the analogue of the Yamabe quotient for our equation.
This motivates the following.

\begin{defi}
We call the convolution $q_u(x):=\frac{1}{|x|^\ell} * u^\frac{ 2n }{n-2}$ the \emph{quotient
of $u$} (which is always non-negative). For a given real number $K>0$, denote 
by $\cQ_K$ the set of solutions $u$ whose quotient is less than
or equal to $K$. More precisely
\begin{gather}
\cQ_K := \Big \{ u \, \Big | \, q_u(x) \leq K \, \text{ for all } x \in \RR^n \Big \}.
\nonumber
\end{gather}
\end{defi}
\begin{rema}
Since $q_u$ is related to the energy on the left-hand side of (\ref{energy}),
$u \in \cQ_K$ can be thought of as an energy-type of condition.
This should not be confused, however, with the more physically appealing notion
of energy for (\ref{Choquard_eq}) used in 
\cite{Cing, ClappSal, MaZhao,  MorozSch}.
\end{rema}

In order to motivate the extra hypotheses that will be needed, we have  to say
a few words about the general situation that will be investigated. We shall 
employ blow-up techniques to study sequences $\{ u_i \}$ of $C^{2,\al}$ solutions
to (\ref{critical_SN_eq}),
and we are primarily concerned with the constraints that the equation imposes on 
blow-up up sequences.
Understanding how a sequence $\{u_i\}$ can be 
unbounded in, say, the $C^0$-norm, 
is important not only because such
families of solutions are obstructions to the application of standard compactness
arguments, but also because this type of behavior is expected
for many critical equations 
(see  \cite{Am1, Berti, CWY, Dr2, DrH, Esp1, Esp2,  HVa} and references
therein).

Whenever blow-up occurs, i.e., $\p u_i \p_{C^0(\RR^n)} \rar \infty$ as $i \rar \infty$,
 condition (\ref{asymptotic})
restricts the blow-up to within a compact set, in which case, we can 
assume $u_i$ to diverge along a sequence of points $x_i \rar \bar{x}$. An analysis of
the sequence $\{ u_i \}$ is carried out by rescaling the solutions
 and the coordinates, leading to an appropriate blow-up model for equation
 (\ref{critical_SN_eq}). In this situation, one expects that
 the blow up of $u_i$, together  with $\{ u_i \} \subset \cQ_K$, implies
 that the rescaled $q_{u_i}$'s are very close to a constant in the neighborhood 
 of $\bar{x}$. But in order to avoid substantial extra work that would 
 distract us from the main goals of the paper, we shall simply assume that $q_{u_i}$ has this
 desired property. Moreover, although our analysis will be local in nature, 
 to avoid the introduction of further cumbersome hypotheses, such an
 assumption will be taken to hold on a big compact set. The precise behavior 
 of $q_u$ has to be ultimately determined by a more refined analysis of the solutions
 to equation (\ref{critical_SN_eq}), what  is beyond the scope of this paper.
 A technical condition on the \emph{rate} at which $q_{u_i}$ approaches a constant
 will also be assumed, although probably this can be relaxed.
 With this in mind we now state our results, whose essence is that 
\emph{control over the convolution $q_u$ yields uniform 
 control over the solutions}. From elliptic theory, one would expect that 
 the required control on $q_u$ should be in the $C^{0,\al}$ topology, and that turns out
 to be in fact the case. 
 
 \begin{theorem}
 Fix  positive numbers $\varrho$, $L$ and $K$. Let 
 $u_i$ be a sequence of $C^{2,\al}$ positive 
 solutions to (\ref{critical_SN_eq}), $0 < \al < 1$,
satisfying
 \begin{gather}
 \{ u_i \}_{i=1}^\infty \subset \cC_{\varrho, L} \cap \cQ_K,
 \nonumber
 \end{gather}
 and suppose  that there exists a constant $Q$ such that
 \begin{gather}
 q_{u_i} \rar Q > 0\, \text{ in } \, C^{0,\al}(B_r(0)) \, \text{ as } \, i \rar \infty,
 \nonumber
 \end{gather}
 for some $r > \varrho$. Suppose further that $n\geq 6$. 
 
 If $\p u_i \p_{C^0(\RR^n)} \rar \infty$  and
 $ \p q_{u_i} - Q \p_{C^{0,\al}(B_r(0))} \p u_i \p_{C^0(\RR^n)}^{n-2} \rar 0$
 as $i \rar \infty$, then, up to a subsequence,
the following holds. There exist $\bar{x} \in B_{\varrho}(0)$, a sequence
$x_i \rar \bar{x}$ and a positive number $\si$  such that 
\begin{gather}
\p u_i \p_{C^0(\RR^n)} \, = u_i(x_i)
\nonumber
\end{gather}
and
\begin{gather}
\p  (u_i(x_i))^{-1} u - (u_i(x_i))^{-1} z_i 
\p_{C^0(B_\si(\bar{x} ) ) } \rar 0 \, \text{ as }\, i \rar \infty, 
\nonumber
\end{gather}
where 
\begin{gather}
z_i(x) := (u_i(x_i))^{-1} \Big ( (u_i(x_i))^{-\frac{4}{n-2}} - 
\frac{Q}{n(n-2)}|x - x_i|^2 \Big )^\frac{2-n}{2}.
\nonumber
\end{gather}
Furthermore, 
 the following estimate holds
 \begin{gather}
\p  (u_i(x_i))^{-1} u - (u_i(x_i))^{-1} z_i 
\p_{C^0(B_\si(\bar{x} ) ) }  \leq
C (u_i(x_i))^{-\frac{4}{n-2}},
\nonumber
 \end{gather}
 where $C=C(L,\varrho, K, Q, r, n, \al, \p V \p_{C^{0,\al}(B_r(0))} )$.
\label{C_0_theorem}
 \end{theorem}
\begin{rema}
 The restriction to $n\geq 6$ is used to obtain the most direct proof without considering
 variations of (\ref{asymptotic}). We believe that this can be removed by a more careful
 application of the techniques here presented.
\end{rema} 
 
 Theorem \ref{C_0_theorem} 
identifies a blow-up model for 
 equation (\ref{critical_SN_eq}). In other words, it states that under
 suitable energy and decay conditions, and up to a subsequence, 
any family of solutions
 that blows up is approximated, after rescaling and 
 near the blow up point $\bar{x}$, by the radially 
 symmetric functions $z_i$. 
The following corollary says that if we also rescale the coordinates, then 
the $C^0$ convergence of theorem \ref{C_0_theorem} is 
improved to $C^2$ convergence.

\begin{theorem} Assume the same hypotheses and notation of theorem \ref{C_0_theorem}.
If $\p u_i \p_{C^0(\RR^n)} \rar \infty$  and
$ \p q_{u_i} - Q \p_{C^{0,\al}(B_r(0))} \p u_i \p_{C^0(\RR^n)}^{n-2} \rar 0$ 
 as $i \rar \infty$, letting $x_i$ and $\bar{x}$ be as in the conclusion of
 theorem \ref{C_0_theorem}, the following holds.
Define
\begin{gather}
v_i(y) := (u_i(x_i))^{-1} u_i(x_i + (u_i(x_i))^\frac{2}{2-n} y)
\nonumber
\end{gather}
and 
\begin{gather}
Z_i (y) := (u_i(x_i))^{-1} z_i(x_i + (u_i(x_i))^\frac{2}{2-n} y).
\nonumber
\end{gather}
Then  $Z_i(y) = (1 + \frac{Q}{n(n-2)}|y|^2)^\frac{2-n}{2} \equiv Z(y)$ for every $i$, and 
\begin{gather}
\p v_i - Z \p_{C^2(B_{\la_i} (0) ) } 
\rar 0 \, \text{ as }\, i \rar \infty,
\nonumber
\end{gather}
where $\la_i = (u_i(x_i))^\frac{2}{n-2} \eta$, with $\eta$ a small positive number.
Furthermore, 
 the following estimate holds
 \begin{gather}
\p v_i - Z \p_{C^2(B_{\la_i} (0) ) } \, \leq \, 
C (u_i(x_i))^{-\frac{4}{n-2}},
\nonumber
 \end{gather}
 where $C=C(L,\varrho, K, Q, r, n, \al, \p V \p_{C^{0,\al}(B_r(0))} )$.
\label{C_2_theorem}
\end{theorem}

The function $Z$ in theorem \ref{C_2_theorem} is well known in conformal geometry.
The metric on the sphere written in a coordinate chart via stereographic 
projection is conformal to the Euclidean metric, with 
the conformal factor being a multiple of $Z$. $Z$ has also 
been  extensively used in the study of the Yamabe
problem.

One naturally wonders about the boundedness of families of solutions to 
(\ref{critical_SN_eq}) as well as the possibility that sequences $u_i > 0$ degenerate
in the limit, i.e., points where $\lim_{i\rar \infty} u_i(x) = 0$.
The next theorem gives some sufficient conditions for bounds from above and below on 
$u$.

\begin{theorem}
Assume the same hypothesis of theorem \ref{C_0_theorem}, 
suppose that $V \neq 0$ and that it 
does not change sign in $B_r(0)$. A sufficient condition
for the existence of a constant $C=C(L,\varrho, K, Q, r, n, \al, \p V \p_{C^{0,\al}(B_r(0))} )$
such that the inequalities
\begin{gather}
\p u \p_{C^0(\RR^n)} + 
\p u \p_{C^{2,\al}(B_\varrho(0) )}  \, \leq \, C,
\nonumber
\end{gather}
and
\begin{gather}
\p u \p_{C^{0}(B_\varrho(0) )} \,  \geq \frac{1}{C}
\nonumber
\end{gather}
hold for any $u \in \cC_{\varrho, L} \cap \cQ_K$, is that
the inequality of theorem \ref{C_2_theorem} be true
in the $C^0$-norm with 
$(u_i(x_i))^{-\frac{4}{n-2} }$ replaced by $(u_i(x_i))^{-\frac{4}{n-2}  - \de}$, for some $\de >0$.
\label{a_priori_bound_theorem}
\end{theorem}

\section{Proof of the theorems}

The hypotheses and notation 
of theorem \ref{C_0_theorem} will be assumed
throughout this section. Let $\{ u_i \}$ be such that 
 $\p u_i \p_{C^0(\RR^n)} \rar \infty$ as $i \rar \infty$.
Since $u$ has $(\varrho,L)$-decay, we can assume 
that for large $i$
\begin{gather}
 \p u_i \p_{C^0(\RR^n)} \, = \, \p u_i \p_{C^0(B_\varrho (0))}.
\nonumber
\end{gather}
Letting $x_i$ be such that $\p u_i \p_{C^0(\overline{B_\varrho(0)})} = u(x_i)$,
and passing to a subsequence we can assume that $x_i \rar \bar{x}$ for some
$\bar{x} \in \overline{B_\varrho(0)}$, and taking $i$ 
sufficiently large if necessary we can also suppose that $x_i$ and $\bar{x}$
are interior points. Define a sequence of real numbers $\{ \ve_i \}_{i=1}^\infty$ by
\begin{gather}
 \ve_i^\frac{2-n}{2} := u_i(x_i).
\nonumber
\end{gather}
Notice that $\ve_i \rar 0$ when $i \rar \infty$. $\ve_i$ measures the rate at which
$u_i$ blows-up. For each fixed $i$, consider the change of coordinates 
\begin{gather}
 y = \ve_i^{-1} (x - x_i),
\nonumber
\end{gather}
and define the rescaled functions 
\begin{gather}
 v_i (y) := \ve^\frac{n-2}{2} u_i( x_i + \ve_i y).
\nonumber
\end{gather}
Then $0$ is a local maximum for $v_i$ with $v_i(0) = 1$ and 
\begin{gather}
 0 < v_i \leq 1.
\label{v_bounded}
\end{gather}
A direct computation shows that $v_i$ satisfies
\begin{gather}
 \Delta v_i + \widetilde{q}_i v_i^\frac{n+2}{n-2} - \ve_i^2 \widetilde{V} v_i = 0,
\label{eq_v}
\end{gather}
where $\widetilde{q}_i(y) =  q_{u_i}(x_i + \ve_i y)$ and 
$\widetilde{V}(y) =  V(x_i + \ve_i y)$.
Let $z_i$ and $Z$ be as in theorems \ref{C_0_theorem} and \ref{C_2_theorem}.
In the sequel, we shall evoke several standard estimates of elliptic theory.
A full account of these results can be found in \cite{GT}.	

\begin{lemma}
With the above definitions, up to a subsequence, it holds that
$v_i \rar Z$ in $C^2_{loc}(\RR^n)$. 
\label{lemma_C_2_loc_v}
\end{lemma}
\begin{proof}
 Fix $R>0$ and $d>0$, and set $R^\prime = R+d$, $R^{\prime\prime} = R + 2d$.
  For any $p>1$, we have by $L^p$ estimates that
\begin{gather}
 \p v_i \p_{W^{2,p}(B_{R^\prime}(0))} \leq C \big ( \p v_i \p_{L^{p}(B_{R^{\prime\prime}}(0))} + 
\p \widetilde{q}_i v_i^\frac{n+2}{n-2} \p_{L^{p}(B_{R^{\prime\prime}}(0))} +  \ve_i^2 \p  \widetilde{V} v_i \p_{L^{p}(B_{R^{\prime\prime}}(0))} \big ),
\nonumber
\end{gather}
where $C=C(n,p,R^\prime, R^{\prime\prime} )$ and $W^{2,p}$ is the usual Sobolev space of functions with $2$ weak 
derivatives in $L^p$. From (\ref{v_bounded}) and our hypothesis it follows that 
\begin{gather}
\p v_i \p_{W^{2,p}(B_{R^\prime}(0))} \, \leq \, C(1 + \ve_i^2) \, \leq \, 2 C,  
\nonumber
\end{gather}
for large $i$, where $C = C(n,p,R^\prime,R^{\prime\prime},K, \p  \widetilde{V} \p_{C^0(B_{R^{\prime\prime}}(0))} )$. 
Choosing $p$ such that $2 > \frac{n}{p}$, we then obtain by the Sobolev embedding theorem that
$v_i$ is bounded in $C^{1,\al^\prime}(B_{R^\prime}(0))$, for some 
$0 < \al^\prime < 1$, therefore there exists a subsequence, 
still denoted $v_i$, which converges in $C^{1,\al}(B_{R^\prime}(0))$, $0 < \al < \al^\prime$, to a limit
$v_\infty$. 

Next, evoke Schauder estimates to obtain
\begin{gather}
 \p v_i \p_{C^{2,\al}(B_R(0))} \leq C \big ( \p v_i \p_{C^{0}(B_{R^\prime}(0))} + 
\p \widetilde{q}_i v_i^\frac{n+2}{n-2} \p_{C^{0,\al}(B_{R^\prime}(0))} 
+  \ve_i^2 \p  \widetilde{V} v_i \p_{C^{0,\al}(B_{R^\prime}(0))} \big ),
\nonumber
\end{gather}
where $C=C(n,\al,R,R^\prime)$. Combining this inequality, the interpolation inequality, the previous bound on
the $C^{1,\al}$-norm of $v_i$ and the hypotheses of theorem \ref{C_0_theorem}, we conclude that 
\begin{gather}
 \p v_i \p_{C^{2,\al}(B_R(0))}  \, \leq \, C,
\nonumber
\end{gather}
with $C = C(n,\al,p,R, d, K, Q, \p  \widetilde{V} \p_{C^0(B_{R^{\prime\prime}}(0))})$. As a consequence,
up to a subsequence and decreasing $\al$ if necessary, the above $C^{1,\al}$ convergence is in fact 
$C^{2,\al}$ convergence, and we can pass to the limit in 
(\ref{eq_v}) to conclude that $v_\infty$ satisfies 
\begin{gather}
 \Delta v_\infty + Q v_\infty^\frac{n+2}{n-2} = 0 \, \text{ in } \, B_R(0).
\nonumber
\end{gather}
Take now a sequence $R_j \rar \infty$. Using the above argument on each $B_{R_j}(0)$
along with a standard diagonal subsequence construction produces a subsequence
$v_i$ that  converges to a limit $v_\infty$ satisfying
\begin{gather}
 \Delta v_\infty + Q v_\infty^\frac{n+2}{n-2} = 0 \, \text{ in } \, \RR^n,
\label{eq_v_R_n}
\end{gather}
with the convergence being in $C^{2,\al}$ on each fixed $B_{R_j}(0)$.
Therefore $v \rar v_\infty$ in $C^{2,\al}_{loc}(\RR^n)$.
Solutions to (\ref{eq_v_R_n}) have been studied by Caffarelli, Gidas and Spruck 
in \cite{CGS}. From their results and the fact $v_\infty(0) =1$, we obtain
\begin{gather}
 v_\infty(y) = \Big (1+\frac{Q}{n(n-2)}|y| \Big )^\frac{2-n}{2} \equiv Z(y).
\nonumber
\end{gather}
\end{proof}

\begin{prop}
There exists a constants $C>0$, independent of $i$, such that 
\begin{gather}
 |v_i - Z|(y) \, \leq \, C \ve_i^{2} \, \text{ for every } \, 
 |y| \, \leq \ve_i^{-1},
\nonumber
\end{gather}
possibly after passing to a subsequence.
\label{main_prop}
\end{prop}
\begin{proof}
 Set
\begin{gather}
 A_i = \max_{|y| \leq \, \ve_i^{-1}} |v_i - Z| = |v_i - Z|(y_i),
\nonumber
\end{gather}
where $y_i$ is defined by this relation as a point where the maximum of
$|v_i - Z|$ is achieved for $|y| \leq \, \ve_i^{-1}$. 
Suppose first that, up to a subsequence, 
$|y_i| \leq \frac{\ve_i^{-1}}{2}$.
If the result is not true, then
there exists a subsequence such that
\begin{gather}
 \frac{\ve_i^2}{A_i} \rar 0 \, \text{ as } \, i \rar \infty.
\label{if_not_true}
\end{gather}
Define $w_i(y) = A_i^{-1}( v_i - Z )(y)$. Then 
it satisfies
\begin{gather}
 \Delta w_i + a_i w_i = 
 \frac{\ve_i^2 }{A_i} \Big (   \widetilde{V} v_i 
 + \frac{ Q - \widetilde{q}_i }{\ve_i^2}  v_i^\frac{n+2}{n-2}  
 \Big ),
\label{equation_w_i}
\end{gather}
where
\begin{gather}
a_i = Q \frac{  v_i^\frac{n+2}{n-2} - Z^\frac{n+2}{n-2} }{v_i - Z}.
\nonumber
\end{gather}
By Taylor's theorem,
\begin{gather}
 v_i^\frac{n+2}{n-2} -  Z^\frac{n+2}{n-2} = \frac{n+2}{n-2} Z^\frac{4}{n-2}(v_i - Z) + O(Z^\frac{6-n}{n-2}|v_i - Z|^2).
\label{Taylor}
\end{gather}

From our hypotheses, expansion (\ref{Taylor}), the $(\varrho,L)$-decay of $u$,
lemma \ref{lemma_C_2_loc_v}, (\ref{if_not_true}) 
and equation (\ref{equation_w_i}), it follows by an argument similar to the proof
of lemma \ref{lemma_C_2_loc_v}, that 
$w_i$ is bounded in $C^2_{loc}(\RR^n)$ and 
\begin{gather}
|a_i(y)| \leq \frac{C}{(1 +|y|)^4},
\label{decay_a}
\end{gather}
for some constant $C>0$ independent of $i$.
Passing to a subsequence, we see that $w_i \rar w_\infty$ in $C^2_{loc}(\RR^n)$, and that 
$w_\infty$ satisfies
\begin{gather}
\Delta w_\infty + \frac{n+2}{n-2} Q Z^\frac{4}{n-2} w_\infty = 0.
\label{eq_w_infinity}
\end{gather}
On the other hand, using the Green's function for the Laplacian with Dirichlet
boundary condition, the representation formula and 
(\ref{decay_a}) show that for $|y| \leq \frac{\ve_i^{-1}}{2}$,
\begin{gather}
|w_i(y) | \leq \frac{C}{1+|y|} + C\frac{\ve_i^2}{A_i},
\label{decay_w_i}
\end{gather}
for some constant $C>0$ independent of $i$. In particular, $w_\infty$ has
the property that 
\begin{gather}
\lim_{|y|\rar \infty} w_\infty(y) = 0.
\label{limit_w}
\end{gather}
Up to a harmless rescaling, solutions to (\ref{eq_w_infinity}) with the property
(\ref{limit_w}) have also been studied by Caffarelli, Gidas and Spruck in \cite{CGS}.
They have the property that $w_\infty \equiv 0$ if
 $w_\infty(0) = 0 = |\nabla w_\infty(0)|$. Recalling that $v_i(0) = 1 $ and that
 $0$ is a local maximum of $v_i$, and noticing that $ Z(0) = 1$, $\nabla Z(0) = 0$,
 we see that $w_i(0) = 0 = |\nabla w_i(0)|$. Therefore, 
 $w_\infty(0) = 0 = |\nabla w_\infty(0)|$ holds, and hence  $w_\infty$ vanishes identically.
 Since $w_i(y_i) =1$ by construction, we must have $|y_i| \rar \infty$, but this 
 contradicts
 (\ref{decay_w_i}) because of (\ref{if_not_true}) and $w_i(y_i) = 1$.
 
It remains to prove the proposition in the case when 
 $|y_i| > \frac{\ve_i^{-1}}{2}$. Notice that from the $(\varrho,L)$-decay of
 $u$, we obtain that $v_i(y) \leq C |y|^\frac{2-n}{2}$ for $|y| \geq \varrho \ve_i^{-1}$,
 while $Z$ obeys the estimate $Z(y) \leq C(1+|y|)^{2-n}$.  From these we get 
 $|v_i - Z|(y_i) \leq C |y_i|^\frac{2-n}{2} \leq C \ve_i^2$ if 
 $|y_i| > \frac{\ve_i^{-1}}{2}$.
\end{proof}

\noindent \emph{Proof of theorem \ref{C_0_theorem}:} Notice that 
\begin{gather}
z_i(x) = \ve_i^\frac{2-n}{2} Z(\ve_i^{-1} x).
\nonumber
\end{gather}
Writing the estimate of  proposition \ref{main_prop} in $x$-coordinates
and recalling the definition of $\ve_i$, 
we obtain 
\begin{gather}
\p  (u_i(x_i))^{-1} u - (u_i(x_i))^{-1} z_i 
\p_{C^0(B_1(x_i ) ) } \leq C \ve_i^2.
\nonumber
\end{gather}
Since $x_i \rar \bar{x}$, choosing $\si >0$ small and $i$ large, we obtain the result.
\cqd

\noindent \emph{Proof of theorem \ref{C_2_theorem}:} For the $C^0$-norm
this is simply the estimate of proposition \ref{main_prop}
written in $x$-coordinates. From the $C^0$ bound, we obtain the $C^2$ bound
by standard elliptic estimates, after considering a smaller ball via
the introduction of $\eta$. \cqd

\noindent \emph{Proof of theorem \ref{a_priori_bound_theorem}:}
We shall first show that $\p u_i \p_{C^0(\RR^n)} \, \leq C$ for a constant
independent of $i$. If this is not the case, then we can find a sequence 
$u_i$ that blows up as in the assumptions of theorem \ref{C_0_theorem}.
In 
light of elliptic theory and shrinking $\eta$ if necessary, we obtain that 
the hypothesis on the $C^0$ decay of $v_i - Z$ then yields the estimate
\begin{gather}
\p v_i - Z \p_{C^2(B_{\la_i} (0) ) } \, \leq \, C \ve_i^{2+\de},
\nonumber
\end{gather}
which then gives
\begin{gather}
\ve_i^\frac{n-2}{2} \p \nabla^m(u_i - z_i) \p_{C^0(B_{\eta}(x_i))}\,  \leq C\ve_i^{2+\de-m}, \, m=0,1,2.
\label{der_estimates}
\end{gather}
Notice that
\begin{gather}
\ve_i^\frac{n-2}{2}  \Delta(u_i - z_i) + Q \ve_i^\frac{n-2}{2} 
\Big( u_i^\frac{n+2}{n-2} - z_i^\frac{n+2}{n-2} \Big)
+ \ve_i^\frac{n-2}{2} \Big( q_{u_i} - Q \Big ) u_i^\frac{n+2}{n-2} = \ve_i^\frac{n-2}{2}  V u_i.
\label{dif_equations}
\end{gather}
From (\ref{der_estimates}),  and the hypotheses on $q_{u_i}$  we conclude that 
\begin{gather}
\ve_i^\frac{n-2}{2} Q\Big( u_i^\frac{n+2}{n-2} - z_i^\frac{n+2}{n-2} \Big)(x_i)
+ \ve_i^\frac{n-2}{2} \Big( q_{u_i} - Q \Big )(x_i) u_i^\frac{n+2}{n-2}(x_i) \rar 0 \, \text{ as } i \rar \infty.
\label{limit_zero}
\end{gather}
(\ref{der_estimates}) also gives
\begin{gather}
\ve_i^\frac{n-2}{2} \Delta(u_i - z_i)(x_i) \rar 0  \, \text{ as } i \rar \infty.
\label{limit_zero_delta}
\end{gather}
But (\ref{limit_zero}) and (\ref{limit_zero_delta}) yield a contradiction with 
(\ref{dif_equations}) since 
\begin{gather}
\ve_i^\frac{n-2}{2} u_i(x_i) = 1 \, \text{ for every } \, i,
\nonumber
\end{gather}
and $V$ is bounded away from zero in $B_\varrho(0)$.

This establishes a bound on the $C^0$-norm of $u_i$. Restricting the
 problem
to $B_r(0)$ and using standard elliptic estimates, produces a bound on the
$C^2$-norm within $B_\varrho(0)$. With the $C^2$ bound at hand, we  write
\begin{gather}
u_i^\frac{n+2}{n-2} = u_i^\frac{4}{n-2} u_i,
\nonumber
\end{gather}
and consider  $u_i^\frac{4}{n-2} $ as a given coefficient. In this 
case, we can treat the equation as a linear equation (for $u_i$) for which 
the Harnack inequality can be applied, producing the desired 
bound from below on $u_i$. \cqd

\end{document}